\documentclass[11pt,reqno]{amsart}

\usepackage{amsmath, amsthm, amssymb}
\usepackage{amsfonts}
\usepackage[ansinew]{inputenc}
\usepackage[dvips]{epsfig}
\usepackage{graphicx}
\usepackage[english]{babel}

\usepackage{cite}
\usepackage{graphicx}
\usepackage{amscd}
\usepackage{bm}
\usepackage{enumerate}

\usepackage{verbatim}
\usepackage{hyperref}
\usepackage{amstext}
\usepackage{latexsym}
\theoremstyle{plain}

\numberwithin{theorem}{section}
\numberwithin{equation}{section}

\newcommand{\average}{{\mathchoice {\kern1ex\vcenter{\hrule height.4pt
width 6pt depth0pt} \kern-9.7pt} {\kern1ex\vcenter{\hrule
height.4pt width 4.3pt depth0pt} \kern-7pt} {} {} }}

\def\R{\mathbb{R}}

\renewcommand{\d}{\delta }

\newcommand{\D }{\Delta }

\newcommand{\G }{\Gamma}

\newcommand{\n }{\nabla }
\newcommand{\vp }{\varphi }

\renewcommand{\O }{\Omega }

\newcommand{\be}{\begin{equation}}
\newcommand{\ee}{\end{equation}}

\newcommand{\de}{\partial}

 \usepackage{mathrsfs}

\renewcommand{\dim}{{\rm dim}\,}

\renewcommand{\epsilon}{\varepsilon}

\newcommand{\Ds}{ (-\D)^s}

\makeatletter
\theoremstyle{plain}
\newtheorem{thm}{\protect\theoremname}
\theoremstyle{remark}
\newtheorem{rem}[thm]{\protect\remarkname}
\theoremstyle{plain}
\newtheorem{lem}[thm]{\protect\lemmaname}
\theoremstyle{plain}
\newtheorem{prop}[thm]{\protect\propositionname}
\theoremstyle{plain}
\newtheorem{cor}[thm]{\protect\corollaryname}

\makeatother

\usepackage{babel}
\providecommand{\corollaryname}{Corollary}
\providecommand{\lemmaname}{Lemma}
\providecommand{\propositionname}{Proposition}
\providecommand{\remarkname}{Remark}
\providecommand{\theoremname}{Theorem}

\begin{document}

\title[Generic properties of eigenvalues of the fractional Laplacian]{Generic properties of   eigenvalues of the fractional Laplacian}

\author{Mouhamed Moustapha Fall}
\address[Mouhamed M. Fall]{African Institute for Mathematical Sciences Senegal, KM2 Route de Joal Mbour 1418}
\email{mouhamed.m.fall@aims-senegal.org}

\author{Marco Ghimenti}
\address[Marco Ghimenti]{Dipartimento di Matematica
Universit\`a di Pisa
Largo Bruno Pontecorvo 5, I - 56127 Pisa, Italy}
\email{marco.ghimenti@unipi.it }

\author{Anna Maria Micheletti}
\address[Anna Maria Micheletti]{Dipartimento di Matematica
Universit\`a di Pisa
Largo Bruno Pontecorvo 5, I - 56127 Pisa, Italy}
\email{a.micheletti@dma.unipi.it }

\author{Angela Pistoia}
\address[Angela Pistoia] {Dipartimento SBAI, Universit\`{a} di Roma ``La Sapienza", via Antonio Scarpa 16, 00161 Roma, Italy}
\email{angela.pistoia@uniroma1.it}

\thanks{The last three authors are partially supported by the group GNAMPA of Istituto Nazionale di Alta Matematica (INdAM). The second author is partially supported by the GNAMPA project ``Modelli nonlineari in presenza di interazioni puntuali"}

\begin{abstract}
We consider the  Dirichlet  eigenvalues of the fractional Laplacian $(-\Delta)^s$, with $s\in (0,1)$, related to a smooth bounded 
domain $\O$. We prove that there exists an arbitrarily small perturbation $\tilde\O=(I+\psi)(\O)$ of the original 
domain such that all  Dirichlet  eigenvalues of the  fractional Laplacian associated to $\tilde\O$ are simple.  As a consequence we obtain that all Dirichlet eigenvalues of the fractional Laplacian on an interval are simple.
In addition, we prove that for a generic choice of parameters all the eigenvalues of some non-local operators are also simple.
 \end{abstract}

\keywords{Eigenvalues, fractional Laplacian, generic properties, simplicity}

\subjclass{35J60, 58C15}
\maketitle

\section{Introduction and statement of the result}

The present paper is concerned with the Dirichlet eigenvalue fractional problem

\begin{equation}
(-\Delta)^{s}\varphi  =\lambda\varphi \text{ in }\Omega,\qquad 
\varphi =0   \text{ in } \mathbb{R}^{n}\smallsetminus\Omega.\label{eq:Pb-0}
\end{equation}

Here $\Omega$ is a bounded $C^{1,1}$ domain in $\mathbb{R}^{n}$ with $n\geq 1$ and
 $(-\Delta)^s $ with $s\in(0,1)$ is the  fractional Laplacian defined, for $u\in C^2_c(\R^n)$, as
\[
(-\Delta)^{s}u=C_{n,s} {\mathtt {P.V.}}\int_{\mathbb{R}^{n}}\frac{u(x)-u(y)}{|x-y|^{n+2s}}dx=C_{n,s}\lim_{\varepsilon\rightarrow0^{+}}\int_{\mathbb{R}^{n}\smallsetminus B_{\varepsilon}(x)}\frac{u(x)-u(y)}{|x-y|^{n+2s}}dx,
\]
where $C_{n,s}:=s4^{s}\frac{\Gamma(s+n/2)}{\pi^{n/2}\Gamma(1-s)}$
is a renormalization constant
and $B_{\varepsilon}(x)$ is the ball
of radius $\varepsilon$ centered in $x$.

To avoid a priori regularity assumptions, we consider the eigenvalue
problem in a weak sense. We consider the space 
\[
\mathcal{H}_{0}^{s}(\Omega):=\left\{ u\in H^{s}(\mathbb{R}^{n})\ :\ u\equiv0\text{ on }\Omega^{c}\right\} ,
\]
where 
\[
H^{s}(\mathbb{R}^{n}):=\left\{ u\in L^{2}(\mathbb{R}^{n})\ :\ \frac{u(x)-u(y)}{|x-y|^{\frac{n}{2}+s}}\in L^{2}(\mathbb{R}^{n}\times\mathbb{R}^{n})\right\} .
\]
On $\mathcal{H}_{0}^{s}(\Omega)$ we consider the quadratic form 
\[
(u,v)\mapsto\mathcal{E}_{s}^{\Omega}(u,v):=\frac{C_{n,s}}{2}\int_{\mathbb{R}^{n}}\int_{\mathbb{R}^{n}}\frac{(u(x)-u(y))(v(x)-v(y))}{|x-y|^{n+2s}}dxdy.
\]
Then, we call $\varphi_{s}\in\mathcal{H}_{0}^{s}(\Omega)$ an eigenfunction
corresponding to the eigenvalue $\lambda$ if 
\[
\mathcal{E}_{s}^{\Omega}(\varphi_{s},v)=\lambda\int_{\mathbb{R}^{n}}\varphi_{s}vdx\ \ \forall v\in\mathcal{H}_{0}^{s}(\Omega).
\]
In the following, to simplify notation, we will omit the renormalization constant $C_{n,s}$.

It is well known (see e.g. \cite{BRS} and the reference therein for
an exhaustive introduction about these topics) that (\ref{eq:Pb})
admits an ordered sequence of eigenvalues 
\[
0<\lambda_{1,s}<\lambda_{2,s}\le\lambda_{3,s}\le\dots\le\lambda_{1,s}\le\dots\rightarrow+\infty.
\]

Since the first eigenvalue is strictly positive, we can endow $\mathcal{H}_{0}^{s}(\Omega)$
with the norm
\[
\|u\|_{\mathcal{H}_{0}^{s}(\Omega)}^{2}=\mathcal{E}_{s}^{\Omega}(u,u).
\]

In the local case, i.e. $s=1$, it is well known (see \cite{M73,M})  that all the eigenvalues are simple for {\em generic} domains $\Omega$.

It is natural to ask if the same results hold true in the non-local case, i.e. $s\in(0,1)$. As far as we know, there are only two results dealing with the simplicity issue.   Very recently, in \cite{DFW}
the authors prove the
simplicity of radial eigenvalues 
in a ball or an annulus.
In \cite{KWM,K},  the authors prove that all the eigenvalues of  the fractional Laplacian    $(-\Delta)^s$ with $s\in [1/2,1)$ in 
  the interval $\Omega=(-1,1)$ are simple.  However, to our knowledge,  the simplicity eigenvalues on an interval  for all $s\in (0,1)$ remains an open problem.  The present paper solves this open question, as a consequence of our main result.

To study domain perturbations we will consider the space
\[
C^{1}(\mathbb{R}^{n},\mathbb{R}^{n}):=\left\{ \psi:\mathbb{R}^{n}\rightarrow\mathbb{R}^{n}\ :\ \psi^{(i)}\text{ continuous and bounded, }i=0,1\right\} 
\]
endowed with the norm 
\[
\|\psi\|_{1}=\sup_{x\in\mathbb{R}^{n}}\max_{i=0,1}|\psi^{(i)}(x)|.
\]
The first question is: if $\bar{\lambda}$ is an eigenvalue of multiplicity
$\nu>1$ of the operator $(-\Delta)_{\Omega}^{s}$ associated with
the domain $\Omega$ with Dirichlet boundary condition, and $U$ is
an interval such that the intersection of the spectrum of $(-\Delta)_{\Omega}^{s}$
with $U$ consist of the only number $\bar{\lambda}$, there exists
a perturbation $\Omega_{\psi}=(I+\psi)(\Omega)$ of the domain $\Omega$ such
that the intersection of the spectrum of $(-\Delta)_{\Omega_{\psi}}^{s}$
with the interval $U$ consists exactly of $\nu$ simple eigenvalues
of $(-\Delta)_{\Omega_{\psi}}^{s}$? Consequently, a second question arises: there exists
a perturbed domain $\Omega_{\psi}=(I+\psi)(\Omega)$ such that {\em all} the eigenvalues of 
 $(-\Delta)_{\Omega_{\psi}}^{s}$ are  simple.

The answer is affirmative and our main result reads as follows.

\begin{thm}
Let  $\Omega$ be a smooth bounded domain with $C^{1,1}$ boundary.  Then for any $\varepsilon >0$ 
there exists $\psi \in C^{1}(\mathbb{R}^{n},\mathbb{R}^{n})$, with $\|\psi\|_{C^1}<\varepsilon$,
 such that all the eigenvalues of  the problem
$$(-\Delta)^{s}\varphi  =\lambda\varphi   \text{ in }\Omega_\psi=(I+\psi)(\Omega),\qquad
\varphi =0   \text{ in } \mathbb{R}^{n}\smallsetminus\Omega_\psi
$$
are simple.

\label{thm:main-0}

\end{thm}

In other words, it can be said that all the eigenvalues of  the problem (\ref{eq:Pb-0})
are simple for {\em generic} domains $\Omega$, where with generic we mean that, 
given a domain $\Omega$, there exists at least an arbitrarily close domain 
$\tilde{\Omega}=(I+\psi)\Omega$  for which all eigenvalues of (\ref{eq:Pb-0}) are simple.  As a consequence of  Theorem \ref{thm:main-0}, we obtain the simplicity of eigenvalues of the fractional laplacian on intervals.
\begin{cor}\label{cor:intervals}
Let $s\in (0,1)$. Then all eigenvalues of  the eigenvalue  problem
$$(-\Delta)^{s}\varphi  =\lambda\varphi \quad   \text{ in }(-1,1),\qquad
\varphi =0  \quad   \text{ in } \mathbb{R}\smallsetminus(-1,1)
$$
are simple.
\end{cor}
 Corollary \ref{cor:intervals}  follows from  Theorem \ref{thm:main-0} which implies that there exists an open interval $\tilde\O$ (a perturbation of an open bounded interval $\O$) such that all its Dirichlet eigenvalues are simple. Since the dimension of the eigenspaces are invariant under scaling and translation,  Corollary \ref{cor:intervals}  follows immediately.\\

In the  spirit of Theorem \ref{thm:main-0}, we obtain a similar result considering Dirichlet eigenvalue fractional problem with nonconstant coefficients of the type

\begin{equation}
(-\Delta)^{s}\varphi +a(x)\varphi =\lambda\varphi \text{ in }\Omega,\qquad 
\varphi =0   \text{ in } \mathbb{R}^{n}\smallsetminus\Omega\label{eq:Pb}
\end{equation}
and
\begin{equation}
 (-\Delta)^{s}\varphi =\lambda \alpha(x)\varphi  \text{ in }\Omega ,\qquad
\varphi =0   \text{ in } \mathbb{R}^{n}\smallsetminus\Omega
,\label{eq:Pb-1}
\end{equation}

where $a,\alpha\in C^0(\mathbb R^n)$. Again, if $(-\Delta)^s +a(x)I$ is a positive operator (e.g. $\min _{\overline\Omega}a>0$ or $\|a\|_{C^0(\Omega)}$ is small enough) or $\min_{\overline\Omega} \alpha>0$,  
from a (fractional analogue) of Rellich's compactness lemma it is quite standard to deduce that there is an unbounded ordered sequence of eigenvalues $(\lambda_i)_{i\in\mathbb N}$
(see \cite{BRS,F} and the references therein) and that each eigenvalue has finite multiplicity and the first one is simple. 

In the local case, simplicity of the eigenvalues with respect to a perturbation of the coefficients where proved in \cite{U} and we are able to show the nonlocal counterpart of this result. In particular, 
we prove that all the eigenvalues of \eqref{eq:Pb} and \eqref{eq:Pb-1} are simple for {\em generic}  functions $a$  and $\alpha $, respectively, in this two results.
\begin{thm}
\label{main}
Let $a\in C^{0}(\mathbb{R}^{n})$ such that $\min_{\overline\Omega}a>0$ or
  $\|a\|_{C^{0}(\Omega)}$ is small enough. For any $\varepsilon>0$ 
  there exists $b\in C^{0}(\mathbb{R}^{n})$, with $\|b\|_{C^0}<\varepsilon$, 
  such that all the eigenvalues of  the problem
$$(-\Delta)^{s}\varphi +\left(a(x)+b(x)\right)\varphi =\lambda\varphi   \text{ in }\Omega,\qquad
\varphi =0   \text{ in } \mathbb{R}^{n}\smallsetminus\Omega
$$
are simple.
\end{thm}

\begin{thm}
\label{main-1}Let $\alpha\in C^{0}(\mathbb{R}^{n})$ such that $\min_{\overline\Omega}\alpha>0$.  
For any $\varepsilon>0$ there exists $\beta\in C^{0}(\mathbb{R}^{n})$, 
with $\|\beta\|_{C^0}<\varepsilon$, such that all the eigenvalues of  the problem
$$(-\Delta)^{s}\varphi  =\lambda \left(\alpha(x)+\beta(x)\right)\varphi  \text{ in }\Omega,\qquad
\varphi =0   \text{ in } \mathbb{R}^{n}\smallsetminus\Omega
$$
are simple.
\end{thm}

The strategy of the proofs of the above theorems relies on an abstract
result which is presented in Section \ref{sec:abstract-result}. In
particular, Theorem \ref{thm:astratto} provides us a so called \emph{splitting condition}, which is crucial 
to find the perturbation term $\psi$ (or $b$,$\beta$) for which
all eigenvalues are simple as claimed in Theorem \ref{thm:main-0} (Th. \ref{main} and Th. \ref{main-1}, respectively).
Throughout the paper we will give a detailed proof of Theorem \ref{thm:main-0}, from Section \ref{domain} 
to Section \ref{Sec.mainproof} while in Section \ref{sec:Thm2} and in Section \ref{sec:Thm3} 
we will only describe the main steps to get Theorem \ref{main} and Theorem \ref{main-1}.

\subsection*{Acknowledgments}
 The authors would like to  thank Matteo Cozzi, Nicola Soave and Enrico Valdinoci for some helpful discussions.
 
\section{Domain perturbations}\label{domain}

In this section we study how a perturbation of the domain affects
the multiplicity of eigenvalues. The main point is, given a
smooth perturbation of the domain of the form $I+\psi$, to introduce, by a
suitable change of variables, the bilinear form $\mathcal{B}_{s}^{\psi}$
in (\ref{eq:Bpsi}) to which we apply the splitting condition of Theorem
\ref{thm:astratto}. The problem of the splitting of the eigenvalues
with respect to domain perturbation was studied for the standard Laplacian
in \cite{He,LM,M,M73}, from which we derive this strategy and which
we refer to for a bibliography on the subject.

For a function $\psi\in C^1(\mathbb{R}^n,\mathbb{R}^n)$, we define
\[
\Omega_{\psi}:=(I+\psi)\Omega.
\]

If  $\|\psi\|_{C^{1}}\le L$ for some $L<1$  then
$(I+\psi)$ is invertible on $\Omega_{\psi}$ with inverse mapping
$(I+\psi)^{-1}=I+\chi$. In the following we always consider $\psi\in C^{1}(\mathbb{R}^{n},\mathbb{R}^{n})$
with $\|\psi\|_{C^{1}}\le L$. Also, we denote $J_{I+\psi}$ as the
Jacobian determinant of the mapping $I+\psi$. Whenever no ambiguity
is possible, we use also the short notation $J_{\psi}:=J_{I+\psi}$.

\begin{rem}
It is well known that, if $\psi$ is sufficiently regular, the following
expansion holds for $\varepsilon$ small
\begin{align*}
J_{I+\varepsilon\psi} & =1+\varepsilon\text{div}\psi+\varepsilon^{2}a_{2}+\dots+\varepsilon^{n}a_{n}
\end{align*}
for suitable $a_{i}$.
\end{rem}

By the change of variables given by the mapping $(I+\psi)$, and denoted
$\tilde{u}(\xi):=u(\xi+\psi(\xi))$, 
we obtain the bilinear form $\mathcal{B}_{s}^{\psi}$ on $\mathcal{H}_{0}^{s}(\Omega)$

\begin{multline}
\mathcal{E}_{s}^{\Omega_{\psi}}(u,v)=\frac{1}{2}\int_{\mathbb{R}^{n}}\int_{\mathbb{R}^{n}}\frac{(u(x)-u(y))(v(x)-v(y))}{|x-y|^{n+2s}}dxdy\\
=\frac{1}{2}\int_{\mathbb{R}^{n}}\int_{\mathbb{R}^{n}}\frac{(\tilde{u}(\xi)-\tilde{u}(\eta))(\tilde{v}(\xi)-\tilde{v}(\eta))}{|\xi-\eta+\psi(\xi)-\psi(\eta)|^{n+2s}}J_{\psi}(\xi)J_{\psi}(\eta)d\xi d\eta\\
=:\mathcal{B}_{s}^{\psi}(\tilde{u},\tilde{v}),\label{eq:Bpsi}
\end{multline}
for $\tilde{u},\tilde{v}\in\mathcal{H}_{0}^{s}(\Omega)$. Notice that $\mathcal{B}_{s}^{0}(\tilde{u},\tilde{v})=\mathcal{E}_{s}^{\Omega}(\tilde{u},\tilde{v})$.

At this point, one can prove by direct computation the following result.
\begin{lem}
Let $\psi\in C^{1}$, and take $\tilde{u}\in\mathcal{H}_{0}^{s}(\Omega)$.
Then 
\[
\mathcal{B}_{s}^{\psi}(\tilde{u},\tilde{u})=\mathcal{E}_{s}^{\Omega_{\psi}}(u,u)\le C_{1}\left[\mathcal{E}_{s}^{\Omega}(\tilde{u},\tilde{u})+\|\tilde{u}\|_{L^{2}(\Omega)}\right]\le C_{2}\mathcal{E}_{s}^{\Omega}(\tilde{u},\tilde{u})
\]
for some positive contants $C_{1},C_{2}$.
\end{lem}
\begin{rem}
Let us define the map 
\begin{align*}
\gamma_\psi&:\mathcal{H}_{0}^{s}(\Omega_{\psi})\rightarrow\mathcal{H}_{0}^{s}(\Omega);\\
\gamma_\psi(u)&:=\tilde{u}(\xi)=u(\xi+\psi(\xi)).
\end{align*}
By the previous lemma we have that, if $\|\psi\|_{C^1}$  is sufficiently small
the following maps are continuous isomorphism
\begin{align*}
\gamma_{\psi} & :\mathcal{H}_{0}^{s}(\Omega_{\psi})\rightarrow\mathcal{H}_{0}^{s}(\Omega)\\
\gamma_{\psi}^{-1}=\gamma_{\chi} & :\mathcal{H}_{0}^{s}(\Omega)\rightarrow\mathcal{H}_{0}^{s}(\Omega_{\psi}).
\end{align*}
In addition $\mathcal{B}_{s}^{\psi}(\tilde{u},\tilde{v})$ is a scalar
product on $\mathcal{H}_{0}^{s}(\Omega)$, and the norm induced by
$\mathcal{B}_{s}^{\psi}(\cdot,\cdot)$ is equivalent to the one induced
by $\mathcal{E}_{s}^{\Omega}(\cdot,\cdot)$.
\end{rem}

It is well known that the embedding $i:\mathcal{H}_{0}^{s}(\Omega)\rightarrow L^{2}(\Omega)$
is compact, so we consider the adjoint operator, with
respect to $\mathcal{E}_{s}^{\Omega}$, 
\[
i^{*}:L^{2}(\Omega)\rightarrow\mathcal{H}_{0}^{s}(\Omega).
\]
The composition 
$E_\O:=(i^{*}\circ i)_{\Omega}:\mathcal{H}_{0}^{s}(\Omega)\rightarrow\mathcal{H}_{0}^{s}(\Omega)$
is selfadjoint, compact, injective with dense image in $\mathcal{H}_{0}^{s}(\Omega)$
and it holds
\begin{equation}
\mathcal{E}_{s}^{\Omega}\left((i^{*}\circ i)_{\Omega}v,u\right)=\int_{\Omega}uv.\label{eq:ii*-0}
\end{equation}

\begin{rem}
\label{rem:i*-0}If $\varphi_{k}\in\mathcal{H}_{0}^{s}(\Omega)$ is
an eigenfunction of the fractional Laplacian with eigenvalue $\lambda_{k}$,
then $\varphi_{k}$ is an eigenfunction of $(i^{*}\circ i)_{\Omega}$
with eigenvalue $\mu_{k}^{\Omega}:=1/\lambda_{k}$. In fact, it holds
\[
\mathcal{E}_{s}^{\Omega}(\varphi_{k},v)=\lambda_{k}\int_{\mathbb{R}^{n}}\varphi_{k}vdx=\int_{\mathbb{R}^{n}}\lambda_{k}\varphi_{k}vdx=\mathcal{E}_{s}^{\Omega}\left(\lambda_{k}(i^{*}\circ i)_{\Omega}\varphi_{k},v\right),
\]
thus $\lambda_{k}(i^{*}\circ i)_{\Omega}\varphi_{k}=\varphi_{k}$.

\end{rem}

We recall two min-max characterizations of eigenvalues $\mu_{k}^{\Omega}$. We have that
\begin{eqnarray*}
\mu_{1}^{\Omega}  :=\sup_{u\in\mathcal{H}_{\Omega}^{s}\smallsetminus0}\frac{\int_{\Omega}u^{2}dx}{\mathcal{E}_{s}^{\Omega}(u,u)};&& 
\mu_{\nu}^\Omega  :=\sup_{\begin{array}{c}
u\in\mathcal{H}_{\Omega}^{s}\smallsetminus0\\
\mathcal{E}_{s}^{\Omega}(u,e_{t})=0\\
t=1,\dots\nu-1
\end{array}}
\frac{\int_{\Omega}u^{2}dx}{\mathcal{E}_{s}^{\Omega}(u,u)};
\end{eqnarray*}
where $(i^{*}\circ i)_{\Omega}e_{t}=\mu_{t}^{\Omega}e_{t}$; equivalently, 
\begin{equation*}
\mu_{\nu}^{\Omega} :=\inf_{V=\left\{ v_{1},\dots,v_{\nu-1}\right\} }\sup_{\begin{array}{c}
u\in\mathcal{H}_{\Omega}^{s}\smallsetminus0\\
\mathcal{E}_{s}^{\Omega}(u,v_{t})=0\\
t=1,\dots\nu-1
\end{array}}\frac{\int_{\Omega}u^{2}dx}{\mathcal{E}_{s}^{\Omega}(u,u)}.
\end{equation*}

By this characterization, and by (\ref{eq:Bpsi}), it is easy to prove the following result 
\begin{lem}
Every eigenvalue \textup{$\mu_{k}$} of the operator $E_{\psi}:=E_{\O_\psi}$ is
continuous at 0 with respect to $\psi\in C^{1}(\mathbb{R}^{n},\mathbb{R}^{n})$.
\end{lem}

Finally, since in Remark \ref{rem:i*-0} we proved that if $\varphi_{k}$
is an eigenfuntion of $(-\Delta)^{s}$ with Dirichlet boundary conditions
on $\Omega_{\psi}$ with eigenvalue $\lambda_{k}$, then ${\varphi}_{k}$
is an eigenfunction of $E_{\psi}$ with eigenvalue $\mu_{k}:=1/\lambda_{k}$,
to obtain the main result of this paper, we study the
multiplicity of the eigenvalues $\mu_{k}$ of the operator $E_{\psi}$. 
For this purpose, in the next section we collect an abstract result which we will apply to the operator $E_{\psi}$.

\section{An abstract result\label{sec:abstract-result}}

We recall a series of abstract results which holds in general in a
Hilbert space $X$ endowed with scalar product $<\cdot,\cdot>_{X}$.
Later, in the paper, we will apply these abstract results to derive
a splitting condition for multiple eigenvalues. The proof of these
results, are contained \cite[Section 2]{M}. However, to make this
paper self contained, we recall them in the appendix.

Let 

\[
F_{ij}:=\left\{ A\in L(X,X)\ :\ \text{codim Im}A=i\text{ and }\dim\ker A=j\right\} 
\]
be the set of Fredholm operator with indices $i$ and $j$ in the
Banach space $L(X,X):=\left\{ A:X\rightarrow X\ :\ A\text{ linear and continuous}\right\} $. 

We show first that $F_{ij}$ is a smooth submanifold of codimension
$ij$ in $L(X,X)$. It is well known that if $A\in F_{ij}$, there
exist closed subspaces $V,W\subset X$ such that 
\[
X=\ker A\oplus V\text{ and }X=W\oplus\text{Im}A.
\]
Let us call $P,Q,\bar{P}$ and $\bar{Q}$ the projector on $\ker A,V,W,\text{Im}A$,
respectively. It holds 
\begin{lem}
\label{lem:deco}We have 
\[
L(X,X)=L\oplus\mathcal{V},
\]
where
\begin{align*}
\mathcal{V}:= & \left\{ T\in L(X,X)\ :\ T(\ker A)\subset\mathrm{Im}A\right\} \\
L:= & \left\{ \bar{P}HP\in L(X,X)\text{ with }H\in L(X,X)\right\} .
\end{align*}
\end{lem}

\begin{proof}
The claim can be showed immediately noticing that $T=\bar{P}TP+\bar{Q}TQ+\bar{P}TQ+\bar{Q}TP$
and that $\bar{Q}TQ+\bar{P}TQ+\bar{Q}TP\in\mathcal{V}$.
\end{proof}
\begin{lem}
\label{lem:varieta}We have that $F_{ij}$ is an analytic submanifold
of $L(X,X)$. In addition, for any $A\in L(X,X)$, the tangent space
in $A$ to $F_{ij}$, $T_{A}F_{ij}=\mathcal{V}$.
\end{lem}

The proof of this result is postponed to appendix. Here we limit ourselves
to give the main idea. Given $A_{0}\in F_{ij}$, and given $H$ such
that $A_{0}+H$ still belongs to $F_{ij}$, it is possible to write
$H=\bar{P}HP+f(V)$ where $V\in\mathcal{V}$ and $f$ is an analytic
function. Then $F_{ij}$ near $A_{0}$ is a smooth graph on $\mathcal{V}$.
\begin{lem}
\label{lem:M}Let $A\in F_{ij}$ such that $\ker A\not\subset\mathrm{Im}A$.
Then
\[
M=\left\{ A+H+\lambda I\in L(X,X)\ :\ \lambda\in\mathbb{R},A+H\in F_{ij}\text{ and }H\text{ suff. small}\right\} 
\]
is an analytic manifold at $A+\lambda I$, and $T_{A+\lambda I}M=\mathcal{V}\oplus\mathrm{Span}<I>$
where $T_{A+\lambda I}M$ is the tangent space in $A+\lambda I$ to
$M$.
\end{lem}

\begin{proof}
By definition of $\mathcal{V}$, we have that $I\in\mathcal{V}$ if
and only if $\ker A\subset\mathrm{Im}A$, which is not possible by
our hypothesis on $A$. Thus, by Lemma \ref{lem:varieta} we have
that $M$ is a ruled manifold and the thesis follows immediately.
\end{proof}
We can recast the previous result considering $T:X\rightarrow X$
a selfadjoint compact operator with an eigenvalue $\bar{\lambda}$
with multiplicity $\nu$. By Riesz theorem we have that $T-\bar{\lambda}I\in F_{\nu\nu}$
and that $\ker(T-\bar{\lambda}I)\cap\mathrm{Im}(T-\bar{\lambda}I)=\left\{ 0\right\} $.
Moreover by Lemma \ref{lem:M} if $U$ is a suitable neighborhood
of $T-\bar{\lambda}I$ we have that 
\[
\tilde{M}=\left\{ \tilde{T}+\lambda I\in L(X,X)\ :\ \lambda\in\mathbb{R}\text{ and }\tilde{T}\in F_{\nu\nu}\cap U\right\} 
\]
is a smooth manifold and $T_{T-\bar{\lambda}I}\tilde{M}=\tilde{\mathcal{V}}\oplus\mathrm{Span}<I>$
where 
\begin{equation}
\tilde{\mathcal{V}}=\left\{ H\in L(X,X)\ :\ H(\ker(T-\bar{\lambda}I))\subset\mathrm{Im}(T-\bar{\lambda}I)\right\} .\label{eq:Vtilde}
\end{equation}
At this point we are in position to enunciate the main result of this
section.
\begin{thm}
\label{thm:astratto}Let $T_{b}:X\rightarrow X$ be a selfadjoint
compact operator which depends smoothly on a parameter $b$ belonging
to a real Banach space $B$. Let $T_{0}=T$ and let $T_{b}$ be Frechet
differentiable in $b=0$. Let Let $x_{1}^{0},\dots,x_{\nu}^{0}$ be
an orthonormal basis for the eigenspace relative to the eigenvalue
$\bar{\lambda}$ of $T$. If $T_{b}\in\tilde{M}$ for all $b$ with
$\|b\|_{C^{0}}$ small, then for all $b$ there exist a $\rho=\rho(b)\in\mathbb{R}$
such that
\begin{equation}
\left\langle T'(0)[b]x_{j}^{0},x_{i}^{0}\right\rangle _{X}=\rho\delta_{ij}\text{ for }i,j=1,\dots,\nu.
\label{eq:spezzamentoastratto}
\end{equation}
\end{thm}

\begin{proof}
By Lemma \ref{lem:M} we have that, if $T_{b}\in\tilde{M}$ for all
$b$, then 
\[
T'(0)[b]\in\tilde{\mathcal{V}}\oplus\mathrm{Span}<I>.
\]
 So, by (\ref{eq:Vtilde}), for all $b$, there exists $\bar{\lambda}(b)\in\mathbb{R}$,
such that 
\[
\left[T'(0)[b]-\bar{\lambda}(b)I\right](\ker(T-\bar{\lambda}I))\subset\mathrm{Im}(T-\bar{\lambda}I),
\]
that is 
\[
\left\langle \left[T'(0)[b]-\bar{\lambda}(b)I\right]x_{j}^{0},x_{i}^{0}\right\rangle _{X}=0
\]
for all $i,j=1,\dots,\nu$, which implies (\ref{eq:spezzamentoastratto}).
\end{proof}
This theorem says that if condition (\ref{eq:spezzamentoastratto})
is fulfilled, then the eigenvalue $\bar{\lambda}(b)$ has still multiplicity
$\nu$ in a neighborhood of $b=0$.

\section{Splitting of a single eigenvalue}\label{Sec.prel}
We recall that $E_{\psi}=(i^{*}\circ i)_{\Omega_{\psi}}$. Also,
by (\ref{eq:ii*-0}), and by the definition of $\tilde{u}$ we have 
\[
\mathcal{E}_{s}^{\Omega_{\psi}}(E_{\psi}u,v)=<u,v>_{L^{2}(\Omega_{\psi})}=
\int_\Omega \tilde{u}\tilde{v}J_\psi.
\]
By the definition of $\mathcal{B}_{s}^{\psi}$, we can rewrite the
previous formula as 
\[
\mathcal{B}_{s}^{\psi}(\gamma_{\psi}E_{\psi}u,\tilde{v})=\mathcal{E}_{s}^{\Omega_{\psi}}(E_{\psi}v,u)
=\int_\Omega \tilde{u}\tilde{v}J_\psi.
\]
Set 
\begin{equation}
T_{\psi}\tilde{u}:=\gamma_{\psi}E_{\psi}\gamma_{\psi}^{-1}\tilde{u},\label{eq:Tpsi}
\end{equation}
we get that $T_{\psi}:\mathcal{H}_{0}^{s}(\Omega)\rightarrow\mathcal{H}_{0}^{s}(\Omega)$
is a compact selfadjoint operator such that
\[
\mathcal{B}_{s}^{\psi}(T_{\psi}\tilde{u},\tilde{v})=\int_\Omega \tilde{u}\tilde{v}J_\psi
\]
for all $\psi$. 
\begin{rem}
One can prove that $T_{\psi}$ and $\mathcal{B}_{s}^{\psi}$ are differentiable in the $\psi$ variable at $0$. Then
it holds
\begin{equation}
\left(\mathcal{B}_{s}^{\psi}\right)'(0)[\psi](T_{0}\tilde{u},\tilde{v})+\mathcal{B}_{s}^{0}(T_{\psi}'(0)[\psi]\tilde{u},\tilde{v})
=\int_\Omega \tilde{u}\tilde{v}\mathrm{div}\psi.
\label{eq:SommaDerB}
\end{equation}
\end{rem}

\begin{lem}\label{lemma20}
Let $\tilde{u},\tilde{v}\in\mathcal{H}^s_0(\O)$ such that $ (-\Delta)^s \tilde{u}, (-\Delta)^s \tilde{v}\in 
C^\alpha_{\textrm{loc}}(\O)\cap L^\infty(\O)$ with $\alpha>(1-2s)_+$. Then
\begin{equation} 
\left(\mathcal{B}_{s}^{\psi}\right)'(0)[\psi](\tilde{u},\tilde{v})=
-\Gamma^2(1+s)\int_{\partial\Omega} \frac{\tilde{u}}{\delta^s}\frac{\tilde{v}}{\delta^s}\psi\cdot N d\sigma 
-\int_\Omega [\nabla \tilde{u}\cdot \psi (-\Delta)^s \tilde{v}+\nabla \tilde{v}\cdot \psi (-\Delta)^s \tilde{u}]dx
\end{equation}
where $\delta(x)=\mathrm{dist}(x,\mathbb{R}^n\smallsetminus\Omega)$ and $N $ is the exterior normal of $\O$.
\end{lem}

\begin{proof}
If $\|\psi\|_{C^{1}}$ is small, by
direct computation we have that 
\begin{multline}
\left(\mathcal{B}_{s}^{\psi}\right)'(0)[\psi](\tilde{u},\tilde{v})=\\
\frac{1}{2}\int_{\mathbb{R}^{n}}\int_{\mathbb{R}^{n}}
\frac{(\tilde{u}(\eta)-\tilde{u}(\xi))(\tilde{v}(\eta)-\tilde{v}(\xi))}{|\xi-\eta|^{n+2s}}
\left\{\mathrm{div}\psi(\xi)+\mathrm{div}\psi(\eta)
-\frac{(n+2s)(\xi-\eta)\cdot(\psi(\xi)-\psi(\eta))}{|\xi-\eta|^2}\right\}d\xi d\eta\\
=\int_{\mathbb{R}^{n}}\int_{\mathbb{R}^{n}}
{(\tilde{u}(\eta)-\tilde{u}(\xi))(\tilde{v}(\eta)-\tilde{v}(\xi))}
K(\xi,\eta)d\xi d\eta,
\label{eq:derB}
\end{multline}
where 
$$
K(\xi,\eta):=\frac{1}{2}\left\{\mathrm{div}\psi(\xi)+\mathrm{div}\psi(\eta)
-\frac{(n+2s)(\xi-\eta)\cdot(\psi(\xi)-\psi(\eta))}{|\xi-\eta|^2}\right\}
\frac 1{|\xi-\eta|^{n+2s}}.
$$
At this point we use the result of Theorem 1.3 of \cite{DFW} which allows to compute integrals
of the form of (\ref{eq:derB}) and we obtain the conclusion.
\end{proof}

We want to apply the previous result to eigenfunctions of $(-\Delta)^s$ on $\Omega$ with 
Dirichlet boundary conditions. We recall that, by Remark \ref{rem:i*-0}, 
this is equivalent to consider eigenfunctions of the operator $T_0$.

\begin{cor}\label{lem:Diff-cB}
Let $u,v\in\mathcal{H}_{0}^{s}(\Omega)$ satisfy 
$ T_0{u}=\frac 1{\lambda_0} {u}$, and $ T_0{v}=\frac 1{\lambda_0} {v}$. Then  we have 
$$
\left(\mathcal{B}_{s}^{\psi}\right)'(0)[\psi](T_0 u, v)=-\frac{\G^2(1+s)}{\lambda_0}\int_{\partial\O}\frac{ u}{\d^s}\frac{ v}{\d^s}\,\psi  \cdot N\,d\sigma+ \int_{\O}  uv\mathrm{div}(\psi)dx.
$$
\end{cor}
\begin{proof} 
By elliptic regularity the eigenfunctions belongs to $C^\alpha_{\textrm{loc}}(\O)\cap L^\infty(\O)$ with $\alpha>(1-2s)_+$. Then, by Lemma \ref{lemma20} we have
\begin{align*}
\left(\mathcal{B}_{s}^{\psi}\right)'(0)[\psi](T_0 u, v)= &-\frac{\G^2(1+s)}{\lambda_0}\int_{\partial\O}\frac{ u}{\d^s}\frac{ v}{\d^s}\,\psi\cdot N\,d\sigma
\\
&-\frac1{\lambda_0}\int_{\O}\n  u\cdot \psi\Ds  v\,dx -\frac1{\lambda_0} \int_{\O}\n  v\cdot\psi\Ds  u\,dx
\end{align*} 
Combining this with Remark \ref{rem:i*-0} and integration by parts, we obtain
\begin{align*}
\left(\mathcal{B}_{s}^{\psi}\right)'(0)&[\psi](T_0 u, v)=
-\frac{\G^2(1+s)}{\lambda_0}\int_{\partial\O}\frac{ u}{\d^s}\frac{ v}{\d^s}\,\psi\cdot N\,d\sigma-\int_{\O}\n  u\cdot \psi  v\,dx-\int_{\O}\n  v\cdot\psi  u\,dx
\\
&=-\frac{\G^2(1+s)}{\lambda_0}\int_{\partial\O}\frac{ u}{\d^s}\frac{ v}{\d^s}\,\psi\cdot N\,d\sigma
 +\int_{\O}  {uv}\text{div}(\psi)dx,
\end{align*}
as desired.
\end{proof}

Now we apply Theorem \ref{thm:astratto} to the operator $T_{\psi}$ defined
in (\ref{eq:Tpsi}). This is the fundamental block to prove Theorem \ref{main}.

Let $\mu_0$ be an eigenvalue of $T_0=E_\O=(i^{*}\circ i)_{\Omega}$ which has multiplicity $\nu>1$. If for all $\psi$ with
$\|\psi\|_{ C^1 }$ small, the operator $T_\psi$ has an eigenvalue $\mu(\psi)$  has 
multiplicity $\nu$ for all $\psi$ and such that $\mu(\psi)\rightarrow \mu_0$ while $\psi\rightarrow 0$, then Theorem \ref{thm:astratto}
yields
\[
\mathcal{B}_{s}^{0}(T_{\psi}'(0)[\psi]\varphi_{i},\varphi_{j})=\rho I
\] 
for some $\rho=\rho(\psi)\in \mathbb{R}$.
 Here $\left\{ \varphi_{i}\right\} _{i=1,\dots,\nu}$ is an orthonormal
base for the eigenspace $\mu(0)$. 
This, in light of (\ref{eq:SommaDerB}) and Corollary \ref{lem:Diff-cB}
can be recast as
\begin{align}\label{eq:split}
\rho \d_{ij}&=-\left(\mathcal{B}_{s}^{\psi}\right)'(0)[\psi](T_{0}\varphi_{i},\varphi_{j})+\int_\Omega
\varphi_{i}\varphi_{j} \mathrm{div}\psi dx \nonumber\\
&={\G^2(1+s)}{\mu_0}\int_{\partial\O}\frac{ \varphi_{i}}{\d^s}\frac{\varphi_{j}}{\d^s}\,\psi  \cdot N\,d\sigma.
\end{align}
So,  for all $\psi$ with
$\|\psi\|_{ C^1 }$ small,
\begin{eqnarray*}
\int_{\partial\O}\frac{ \varphi_{i}}{\d^s}\frac{\varphi_{j}}{\d^s}\,\psi  \cdot N\,d\sigma=0\text{ for }i\neq j;&&
\int_{\partial\O}\left(\frac{ \varphi_{1}}{\d^s}\right)^2\,\psi  \cdot N\,d\sigma=\dots=
\int_{\partial\O}\left(\frac{ \varphi_{\nu}}{\d^s}\right)^2\,\psi  \cdot N\,d\sigma.
\end{eqnarray*}
This implies that $(\frac{\vp_i}{\d^s})^2\equiv 0$ on $\de\O$ for $i=1,\dots,\nu$.  
On the other hand, by the fractional Pohozaev identity (see \cite{ROS} and \cite[formula (1.6)]{DFW}),
\begin{equation*}
{\G^2(1+s)}\int_{\partial\O}\left(\frac{\varphi_i}{\d^s}\right)^2 x \cdot N\,d\sigma=
\frac{2s}{\mu_0}\int_{\O}\varphi_i^2dx = \frac{2s}{\mu_0}\not=0.
\end{equation*}
This leads to a contradiction and thus    $T_\psi$ cannot have multiplicity $\nu$ for all $\psi$ with
$\|\psi\|_{ C^1 }$ small.
This fact can be summarized in the next proposition, 
which is the main tool to prove Theorem \ref{thm:main-0}.

\begin{prop}
\label{thm:main-tool}Let $\bar{\lambda}$ an eigenvalue of
the operator $(-\Delta)_{\Omega}^{s}$ with Dirichlet boundary condition which has
multiplicity $\nu>1$. Let $U$ and open bounded interval such
that 
\[
\bar{U}\cap\sigma\left((-\Delta)_{\Omega}^{s}\right)=\left\{ \bar{\lambda}\right\} ,
\]
where $\sigma\left((-\Delta)_{\Omega}^{s}\right)$ is the spectrum
of $(-\Delta)_{\Omega}^{s}$. 

Then, there exists $\psi\in C^{1}(\mathbb{R}^{n},\mathbb{R}^{n})$
such that for $\ensuremath{\Omega_{\psi}=(I+\psi)\Omega}$ it holds
\[
\bar{U}\cap\sigma\left((-\Delta)_{\Omega_{\psi}}^{s}\right)=\left\{ \lambda_{1}^{\Omega_{\psi}},\dots,\lambda_{k}^{\Omega_{\psi}}\right\} ,
\]
where $\lambda_{i}^{\Omega_{\psi}}$ is an eigenvalue of the operator
$(-\Delta)_{\Omega_{\psi}}^{s}$ associated to the set $\Omega_{\psi}$
with Dirichlet boundary condition. Here $k>1$ and the multiplicity
of $\lambda_{i}^{\Omega_{\psi}}$ is $\nu_{i}$ with $\sum_{i=1}^{k}\nu_{i}=\nu$.
\end{prop}

We recall that if $\|\psi\|_{C^1}$ is small, the multiplicity of an eigenvalue $\lambda^{\O_\psi}$ near $\bar{\lambda}$ can only be equal or smaller than the multiplicity of $\bar{\lambda}$. Here, in Proposition \ref{thm:main-tool}, 
we proved the existence of perturbations for which the multiplicity is strictly smaller.

The next corollary follows from Proposition \ref{thm:main-tool}, composing
a finite number of perturbations.
\begin{cor}\label{cor:auto-semplice-0}
There exists $\psi\in C^{1}(\mathbb{R}^{n},\mathbb{R}^{n})$ such
that for $\ensuremath{\Omega_{\psi}=(I+\psi)\Omega}$ it holds
\[
\bar{U}\cap\sigma\left((-\Delta)_{\Omega_{\psi}}^{s}\right)=\left\{ \lambda_{1}^{\Omega_{\psi}},\dots,\lambda_{\nu}^{\Omega_{\psi}}\right\},
\]
where $\lambda_{i}^{\Omega_{\psi}}$ is a simple eigenvalue of the
operator $(-\Delta)_{\Omega_{\psi}}^{s}$ associated to the set $\Omega_{\psi}$
with Dirichlet boundary condition.
\end{cor}
At this point we are in position to prove the main result of this paper.

\section{Proof of Theorem \ref{thm:main-0}}\label{Sec.mainproof}

We start proving the following splitting property for a finite number
of multiple eigenvalues.
\begin{lem}
\label{thm:itera-finita}Given a sequence $\left\{ \sigma_{l}\right\} $
of positive real numbers there exists
\begin{itemize}
\item a sequence of bijective map  $\left\{ F_{l}\right\} \in C^{1}(\mathbb{R}^{n},\mathbb{R}^{n})$, $F_l=(I+\psi_l)$
with $\|\psi_{l}\|_{C^{1}}\le\sigma_{l}$
\item a sequence of open bounded $C^1$ sets with $\Omega_0=\Omega$ and $\Omega_l=F_l(\Omega_{l-1})$ 
\item a sequence of increasing integer numbers $\left\{ q_{l}\right\} $
with $q_{l}\nearrow+\infty$
\item a sequence of open bounded intervals $\left\{ U_{t}\right\} _{t=1,\dots,q_{l}}$
with $\bar{U}_{i}\cap\bar{U}_{j}=\emptyset$ for $i\neq j$
\end{itemize}
such that the eigenvalues $\lambda_{i}^{\O_l}$ of
the operator $(-\Delta)_{\Omega_{l}}^{s}$ are
simple for $i=1,\dots,q_{l}$ and $\lambda_{i}^{\O_l}\in U_{i}$
for all $i=1,\dots,q_{l}$.
\end{lem}

\begin{proof}
Take $q\in\mathbb{N}$ such that that $\lambda_{1},\dots,\lambda_{q}$
are simple eigenvalues for $(-\Delta)_{\Omega}^{s}$ and that $\lambda_{q+1}$
is the first eigenvalue with multiplicity $\nu_{q+1}$. For $t=1,\dots,q$
let $ U_{t} $ be open intervals such that $\bar{U}_{i}\cap\bar{U}_{j}=\emptyset$
for $i\neq j$ and $\lambda_{t}\in U_{t}$. Let us take $W$ an
open interval such that $\bar{W}\cap\bar{U}_{t}=\emptyset$ for all
$t=1,\dots,q$ and $\bar{W}\cap\sigma((-\Delta)_{\Omega}^{s})=\left\{ \lambda_{q+1}\right\} $.
At this point, by Corollary \ref{cor:auto-semplice-0} we can choose
\textbf{$\bar{\psi}$ }such that $\bar{W}\cap\sigma((-\Delta)_{\Omega_{\bar\psi}}^{s})I$
contains exactly $\nu_{q+1}$ simple eigenvalues. Also, we can choose
a number $\sigma_{q+1}$ sufficiently small, with $\|\bar{\psi}\|_{C^{1}}\le\sigma_{q+1}$
so that $\lambda_{t}^{\bar{\psi}}\in U_{t}$ for all $t=1,\dots,q$,
since the eigenvalues depends continuously on $\psi$. At this point,
by iterating this procedure a finite number of times we get the proof. 
\end{proof}
At this point we are in position to prove the first result of our
paper

\begin{proof}[Proof of Theorem \ref{thm:main-0}]
 Let us take a sequence $\left\{ \sigma_{l}\right\} $ with $0<\sigma_{l}<\frac{1}{4^{l}}$,
and a sequence $F_l=(1+\psi_{l})$ associated to $\sigma_{l}$ as in the previous
theorem.
We set 
$$
\mathcal{F}_l=F_l\circ F_{l-1}\circ \dots \circ F_1.
$$
We can prove that, by the choice of $\sigma_{l}$, the sequence $\{\mathcal{F}_l-I\}_l$ converges to some function 
$\bar{\psi}$ in $C^{1}(\mathbb{R}^{n},\mathbb{R}^{n})$. 
In fact, by the previous lemma we have
\begin{eqnarray}
\|\mathcal{F}_{i+1}-\mathcal{F}_{i}\|_\infty&\le&\|\psi_{i+1}\|_{C^1}<\left(\frac14\right)^{i+1}\\
\|\mathcal{F}'_{i+1}-\mathcal{F}'_{i}\|_\infty&\le&\|\psi_{i+1}\|_{C^1}\|\mathcal{F}'_{i}\|_\infty
\le\left(\frac14\right)^{i+1}\|\mathcal{F}'_{i}\|_\infty.
\label{derF}
\end{eqnarray}
By induction, using \ref{derF}, we can prove that
\begin{equation}
\|\mathcal{F}'_{i}\|_\infty\le \left(1+\frac14\right)^{i}\le \left(\frac54\right)^{i}
\end{equation}
and, combining all these equation, that
\begin{eqnarray}
\|\mathcal{F}_{i+1}-\mathcal{F}_{i}\|_{C^1}&\le&\|\psi_{i+1}\|_{C^1}\le\left(\frac14\right)^{i+1}\left(\frac54\right)^{i}
\end{eqnarray}
and, by iterating, that, for all $p\in \mathbb{N}$
\begin{eqnarray}
\|\mathcal{F}_{i+p}-\mathcal{F}_{i}\|_{C^1}&\le&\|\psi_i\|_{C^1}
\le\sum_{t=0}^p\left(\frac14\right)^{i+t+1}\left(\frac54\right)^{i+t}\nonumber\\
&\le& \frac14\left(\frac5{16}\right)^{i}\sum_{t=0}^p\left(\frac5{16}\right)^{t}\rightarrow 0 \text{ as }i\rightarrow\infty.
\label{serie}
\end{eqnarray}
Thus the sequence $\{\mathcal{F}_{i}-I\}$ converges in $C^1$ to some $\bar{\psi}=\bar{\mathcal{F}}-I$ and, by (\ref{serie}), 
$\|\bar{\psi}\|_{C^1}\le 1/2$, so $\bar{\mathcal{F}}$ is invertible.

We claim
that all the eigenvalues $(-\Delta)_{\Omega_{\bar{\psi}}}^{s}$ are simple.
By contradiction, suppose that there exists a $\bar{q}$ such that
$\lambda_{\bar{q}}^{\bar{\psi}}$ is the first multiple eigenvalue. Let us call $\O_l=\mathcal{F}_l(\O)$ and 
$\{\lambda^{\O_l}_i\}_i$ the eigenvalues of $(-\Delta)_{\Omega_{l}}^{s}$  on $\O_l$ with Dirichlet boundary conditions. 
By Theorem \ref{thm:itera-finita} we have that there exists an $l\in\mathrm{N}$ such that 
$(-\Delta)_{\Omega_{l}}^{s}$ has the first $\bar{q}+1$ eigenvalues simple, and that there exists
$U_{1},\dots,U_{\bar{q}+1}$ open intervals, with disjoint closure,
such that $\lambda_{t}^{\O_l}\in U_{t}$
for $t=1,\dots,\bar{q}+1$. On the one hand, 
$\lambda_{\bar{q}}^{\O_N}\rightarrow\lambda_{\bar{q}}^{\bar{\psi}}$
as well as $\lambda_{\bar{q}+1}^{\O_N}\rightarrow\lambda_{\bar{q}}^{\bar{\psi}}$
when $N\rightarrow\infty$ by continuity of the eigenvalues. On the
other hand, $\lambda_{\bar{q}}^{\O_N}\in U_{\bar{q}}$
and $\lambda_{\bar{q}+1}^{\O_N}\in U_{\bar{q}+1}$
for all $N$, by Theorem \ref{thm:itera-finita}. So 
$\lambda_{\bar{q}}^{\bar{\psi}}=\lambda_{\bar{q}+1}^{\bar{\psi}}\in\bar{U}_{\bar{q}}\cap\bar{U}_{\bar{q}+1}$
which leads us to a contradiction, and the theorem is proved.
\end{proof}

\section{Proof of Theorem \ref{main}}\label{sec:Thm2}
In this case we call
\[
\mathcal{B}^{a}(u,v)=\mathcal{E}(u,v)+\int_{\mathbb{R}^{n}}au^{2}dx.
\]
and, by the hyphothesis on $a$, we can endow $\mathcal{H}_{0}^{s}(\Omega)$
with the norm 
\[
\|u\|_{\mathcal{H}_{0}^{s}(\Omega)}^{2}=\mathcal{B}^{a}(u,u)=\mathcal{E}(u,u)+\int_{\mathbb{R}^{n}}au^{2}dx.
\]
We call $\varphi^{a}\in\mathcal{H}_{0}^{s}(\Omega)$ an eigenfunction of $\left((-\Delta)^s+a\right)$
corresponding to the eigenvalue $\lambda^{a}$. 
Given the embedding $i:\mathcal{H}_{0}^{s}(\Omega)\rightarrow L^{2}(\Omega)$
we consider its adjoint operator, with
respect to the scalar product $\mathcal{B}^{a}$, 
\[
i^{*}:L^{2}(\Omega)\rightarrow\mathcal{H}_{0}^{s}(\Omega).
\]
It holds
\begin{equation}
\mathcal{B}^{a}\left((i^{*}\circ i)_{a}u,v\right)=\mathcal{E}\left((i^{*}\circ i)_{a}u,v\right)+\int_{\Omega}au(i^{*}\circ i)_{a}v=\int_{\Omega}uv,\label{eq:ii*}
\end{equation}
and, as before, if $\varphi_{k}^{a}\in\mathcal{H}_{0}^{s}(\Omega)$
is an eigenfunction of the fractional Laplacian with eigenvalue $\lambda_{k}^{a}$,
then $\varphi_{k}^{a}$ is an eigenfunction of $(i^{*}\circ i)_{a}$
with eigenvalue $\mu_{k}^{a}:=1/\lambda_{k}^{a}$.

In addiction (\ref{eq:Pb})
admits an ordered sequence of eigenvalues 
\[
0<\lambda_{1}^{a}<\lambda_{2}^{a}\le\lambda_{3}^{a}\le\dots\le\lambda_{k}^{a}\le\dots\rightarrow+\infty
\]
and all the eigenvalues
$\lambda_{k}^{a}$ depends continuously on $a$.

In the following, for $b\in C^{0}(\Omega)$ with $\|b\|_{L^{\infty}}$
small enough we consider $\mathcal{B}^{a+b}$ and $(i^{*}\circ i)_{a+b}$
and we put
\begin{equation}
B_{b}:=\mathcal{B}^{a+b}\text{ and }E_{b}:=(i^{*}\circ i)_{a+b}.\label{eq:notation}
\end{equation}
Similarly to what we proved in Section \ref{Sec.prel} we have the following lemma.
\begin{lem}
\label{lem:Bprimo}The maps $b\mapsto B_{b}$ and $b\mapsto E_{b}$ are differentiable
at $0$ and it holds
\[
(B'(0)[b]u,v)=\int_{\Omega}buv,
\]
\begin{equation}
0=\left(B'(0)[b]E_{0}u,v\right)+B_{0}\left(E'(0)[b]u,v\right).\label{eq:chain}
\end{equation}
for all $u,v\in\mathcal{H}_{0}^{s}(\Omega)$.
\end{lem}

\begin{rem}
Notice that, by Lemma \ref{lem:Bprimo} and by (\ref{eq:chain}),
it holds
\[
-B_{0}\left(E'(0)[b]u,v\right)=\left(B'(0)[b]E_{0}u,v\right)=\int_{\Omega}b(E_{0}u)v=\int_{\Omega}b\left[(i^{*}\circ i)_{a}u\right]v.
\]
\end{rem}

\begin{rem}
\label{rem:spezz}If $\mu^{a}=\mu$ is an eigenvalue of the map $E_{0}=(i^{*}\circ i)_{a}$
with multiplicity $\nu>1$, and $\varphi_{1}^{a},\dots,\varphi_{\nu}^{a}$
are orthonormal eigenvectors associated to $\mu$, then, by the previous
remark we have
\[
\left(B'(0)[b]E_{0}\varphi_{i}^{a},\varphi_{j}^{a}\right)=\int_{\Omega}bE_{0}(\varphi_{i}^{a})\varphi_{j}^{a}=-\mu\int_{\Omega}b\varphi_{i}^{a}\varphi_{j}^{a},
\]
 for all $i,j=1,\dots,\nu$.
\end{rem}

Now we apply the condition (\ref{eq:spezzamentoastratto})
to prove the splitting property for a chosen multiple eigenvalue.

\begin{prop}
\label{thm:main-tool-pert}Let $a\in C^{0}(\mathbb{R}^{n})$ be positive
on $\Omega$ or with $\|a\|_{C^{0}(\Omega)}$ sufficiently small.
Let $\bar{\lambda}$ an eigenvalue of the operator $(-\Delta)_{\Omega}^{s}+aI$
on $\mathcal{H}_{0}^{s}$ with Dirichlet boundary condition with multiplicity
$\nu>1$. Let $U$ and open bounded interval such that 
\[
\bar{U}\cap\sigma\left((-\Delta)_{\Omega}^{s}+aI\right)=\left\{ \bar{\lambda}\right\} ,
\]
where $\sigma\left((-\Delta)_{\Omega}^{s}+aI\right)$ is the spectrum
of $(-\Delta)_{\Omega}^{s}+aI$. 

Then, there exists $b\in C^{0}(\mathbb{R}^{n})$ such that for 
\[
\bar{U}\cap\sigma\left((-\Delta)_{\Omega}^{s}+(a+b)I\right)=\left\{ \lambda_{1}^{b},\dots,\lambda_{k}^{b}\right\} ,
\]
where $\lambda_{i}^{b}$ is an eigenvalue of the operator $(-\Delta)_{\Omega}^{s}+(a+b)I$.
Here $k>1$ and the multiplicity of $\lambda_{i}^{b}$ is $\nu_{i}$
with $\sum_{i=1}^{k}\nu_{i}=\nu$.
\end{prop}

The next corollary follows from by the previous propositon, after
composing a finite number of perturbations.
\begin{cor}
\label{cor:auto-semplice}There exists $b\in C^{0}(\mathbb{R}^{n})$
such that
\[
\bar{U}\cap\sigma\left((-\Delta)_{\Omega}^{s}+(a+b)I\right)=\left\{ \lambda_{1}^{b},\dots,\lambda_{\nu}^{b}\right\},
\]
where $\lambda_{i}^{b}$ is a simple eigenvalue of the operator $(-\Delta)_{\Omega}^{s}+(a+b)I$
with Dirichlet boundary condition.
\end{cor}

\begin{proof}[Proof of Proposition \ref{thm:main-tool-pert}]
We apply Theorem \ref{thm:astratto} to the operator $E_{b}=(i^{*}\circ i)_{a+b}$
introduced in (\ref{eq:notation}). 

If $\mu^{a+b}$ is an eigenvalue of $E_{b}$ which has multiplicity
$\nu$ at $b=0$ and at any $b$ with $\|b\|_{C^{0}}$ small, then
by condition (\ref{eq:spezzamentoastratto}) of Theorem \ref{thm:astratto}
we have 
\[
B_{0}(E'(0)[b]\varphi_{i},\varphi_{j})=\rho\delta_{ij}\text{ for some }\rho\in\mathbb{R},
\]
where $\left\{ \varphi_{i}\right\} _{i=1,\dots,\nu}$ is an $L^{2}$-orthonormal
basis for the eigenspace relative to $\mu^{a}$. Then, in light of Remark \ref{rem:spezz}, we should have
that for any $b\in C^{0}$ small, there exists $\rho=\rho(b)$ such
that 
\[
\mu^{a}\int_{\Omega}b\varphi_{i}\varphi_{j}=\rho(b)\delta_{ij}.
\]
Then, in particular, we deduce that 
\[
\int_{\Omega}b\varphi_{1}\varphi_{2}=0\text{ and}\int_{\Omega}b\varphi_{1}^{2}=\int_{\Omega}b\varphi_{2}^{2}\text{ for all }b\in C^{0}.
\]
Thus $\varphi_{1}\varphi_{2}\equiv0$ and $\varphi_{1}^{2}\equiv\varphi_{2}^{2}$
almost everywhere in $\Omega$. Thus $\varphi_{1}\equiv\varphi_{2}\equiv0$
a.e. in $\Omega$, which leads us to a contradiction. Then there exists
$b\in C^{0}$ small such that the multiplicity of $\mu^{a+b}$ is
smaller that $\nu$. Since the eigenvalue $\mu^{a+b}$ depends continuosly
on $b$, given a neighborhood $U$ of $\mu^{a}$, for $\|b\|_{C^{0}}$
small we have that $\bar{U}\cap\sigma(E_{b})=\left\{ \mu_{1}^{a+b},\dots,\mu_{k}^{a+b}\right\} $
with $\nu_{i}$ the multiplicity of $\mu_{i}^{a+b}$, and where $\sum_{i=1}^{k}\nu_{i}=\nu$,
and $k>1$. Remebering the definition of $E_{b}$ and that $\mu^{a+b}=1/\lambda^{a+b}$
we have the claim.
\end{proof}
We proceed similarly as the proof of Theorem \ref{thm:main-0} to obtain Theorem  \ref{main}

\begin{lem}
\label{thm:itera-finita2}Given $a\in C^{0}(\mathbb{R}^{n})$ as in
the hypotesis of Theorem \ref{main}, and a sequence $\left\{ \sigma_{l}\right\} $
of positive real numbers there esists
\begin{itemize}
\item a sequence of functions $\left\{ b_{l}\right\} \in C^{0}(\mathbb{R}^{n})$
with $\|b_{l}\|_{C^{0}}\le\sigma_{l}$
\item a sequence of increasing integer numbers $\left\{ q_{l}\right\} $
with $q_{l}\nearrow+\infty$
\item a sequence of open bounded intervals $\left\{ U_{t}\right\} _{t=1,\dots,q_{l}}$
with $\bar{U}_{i}\cap\bar{U}_{j}=\emptyset$ for $i\neq j$
\end{itemize}
such that the eigenvalues $\lambda_{i}^{a+\sum_{j=i}^{l}b_{j}}$ of
the operator $(-\Delta)_{\Omega}^{s}+(a+\sum_{j=i}^{l}b_{j})I$ are
simple for $i=1,\dots,q_{l}$ and $\lambda_{i}^{a+\sum_{j=i}^{l}b_{j}}\in U_{i}$
for all $i=1,\dots,q_{l}$.
\end{lem}

\begin{proof}
Take $q\in\mathbb{N}$ such that that $\lambda_{1}^{a},\dots,\lambda_{q}^{a}$
are simple eigenvalues for $(-\Delta)_{\Omega}^{s}+aI$ and that $\lambda_{q+1}^{a}$
is the first eigenvalue with multiplicity $\nu_{q+1}$. For $t=1,\dots,q$
let $\left\{ U_{t}\right\} $ open intervals such that $\bar{U}_{i}\cap\bar{U}_{j}=\emptyset$
for $i\neq j$ and $\lambda_{t}^{a}\in U_{t}$. Let us take $W$ an
open interval such that $\bar{W}\cap\bar{U}_{t}=\emptyset$ for all
$t=1,\dots,q$ and $\bar{W}\cap\sigma((-\Delta)_{\Omega}^{s}+aI)=\left\{ \lambda_{q+1}^{a}\right\} $.
At this point, by Corollary \ref{cor:auto-semplice} we can choose
\textbf{$\bar{b}$ }such that $\bar{W}\cap\sigma((-\Delta)_{\Omega}^{s}+(a+\bar{b})I$
contains exactly $\nu_{q+1}$ simple eigenvalues. Also, we can choose
a number $\sigma_{q+1}$ sufficiently small, with $\|b_{q+1}\|_{C^{0}}\le\sigma_{q+1}$
so that $\lambda_{t}^{a+\bar{b}}\in U_{t}$ for all $t=1,\dots,q$,
since the eigenvalues depends continuosly on $b$. At this point,
by iterating this procedure a finite number of times we get the proof. 
\end{proof}
At this point we can conclude.
\begin{proof}[Proof of Theorem \ref{main}]
 Let us take a sequence $\left\{ \sigma_{l}\right\} $ with $0<\sigma_{l}<\frac{1}{2^{l}}$,
and a sequence $b_{l}$ associated to $\sigma_{l}$ as in the previous
theorem. By the choice of $\sigma_{l}$, we have that $\sum_{l}b_{l}$
converge to some function $b$ in $C^{0}(\mathbb{R}^{n})$. We claim
that all the eigenvalues $(-\Delta)_{\Omega}^{s}+(a+b)I$ are simple.
By contradiction, suppose that there exists a $\bar{q}$ such that
$\lambda_{\bar{q}}^{a+b}$ is the first multiple eigenvalue. By Theorem
\ref{thm:itera-finita} we have that $(-\Delta)_{\Omega}^{s}+(a+\sum_{l=1}^{\bar{q}+1}b_{l})I$
has the first $\bar{q}+1$ eigenvalues simple, and that there exists
$U_{1},\dots,U_{\bar{q}+1}$ open intervals, with disjoint closure,
such that $\lambda_{t}^{a+\sum_{l=1}^{\bar{q}+1}b_{l}}\in U_{t}$
for $t=1,\dots,\bar{q}+1$. On the one hand, $\lambda_{\bar{q}}^{a+\sum_{l=1}^{N}b_{l}}\rightarrow\lambda_{\bar{q}}^{a+b}$
as well as $\lambda_{\bar{q}+1}^{a+\sum_{l=1}^{N}b_{l}}\rightarrow\lambda_{\bar{q}}^{a+b}$
when $N\rightarrow\infty$ by continuity of the eigenvalues. On the
other and, $\lambda_{\bar{q}}^{a+\sum_{l=1}^{N}b_{l}}\in U_{\bar{q}}$
and $\lambda_{\bar{q}+1}^{a+\sum_{l=1}^{N}b_{l}}\in U_{\bar{q}+1}$
for all $N$, by Theorem \ref{thm:itera-finita}. So $\lambda_{\bar{q}}^{a+b}=\lambda_{\bar{q}+1}^{a+b}\in\bar{U}_{\bar{q}}\cap\bar{U}_{\bar{q}+1}$
which lead as to a contradiction, and the theorem is proved.
\end{proof}

\section{Sketch of the proof of Theorem \ref{main-1}.}\label{sec:Thm3}

In this section we adapt the abstract scheme to the last of the second
result of this paper. Since the proof is very similar to the one of Theorem \ref{main}, we provide only the main tools.

Since $\alpha>0$ on $\bar{\Omega}$, we endow the space $L^{2}(\Omega)$
with scalar product and norm given, respectively, by
\[
\langle u,v\rangle_{L^{2}}=\int_{\Omega}\alpha uv;\ \ \ \ \ \|u\|_{L^{2}}^{2}=\int_{\Omega}\alpha u^{2},
\]
 while on $\mathcal{H}_{0}^{s}$ we consider the usual scalar product
$\mathcal{E}(u,v)$. We consider the embedding $i:\mathcal{H}_{0}^{s}\rightarrow L^{2}$
and its adjoint operator $i^{*}:L^{2}\rightarrow\mathcal{H}_{0}^{s}$.
Then we have 
\[
\mathcal{E}((i^{*}\circ i)_{\alpha }v,u)=\int_{\Omega}\alpha uv\ \ \forall u,v\in\mathcal{H}_{0}^{s}.
\]
As before, the map $(i^{*}\circ i)_{\alpha }$ is selfadjoint, compact and
injective form $\mathcal{H}_{0}^{s}$ in itself. In addition, is $\varphi^{\alpha }$
is an eigenfunction associated to the eigenvalue $\mu^{\alpha }$ for $(i^{*}\circ i)_{\alpha }$,
then 
\[
 \mu^{a}(-\Delta)^{s}\varphi =\alpha (x)\varphi_{s}  \text{ in }\Omega,\
\varphi =0  \text{ in } \mathbb{R}^{n}\smallsetminus\Omega
,\]
thus $\lambda^{\alpha }=1/\mu^{\alpha }$ is an eigenvalue with $\varphi^{\alpha }$
as eigenvector for Problem (\ref{eq:Pb-1}). 

We want to prove that there exists $\beta\in C^{0}(\Omega)$, with $\|\beta\|_{L^{\infty}}$
sufficiently small, such that $(i^{*}\circ i)_{\alpha+\beta}$ has all eigenvalues
simple.

Set 
\[
E_{\beta}:=(i^{*}\circ i)_{\alpha +\beta},
\]
we have the following Lemma
\begin{lem}
The map $\beta\mapsto E_{\beta}$ from a neighborhood of $0$ in $C^{0}(\Omega)$
to the space of linear maps from $\mathcal{H}_{0}^{s}(\Omega)$ to
$\mathcal{H}_{0}^{s}(\Omega)$ is continuous and differentiable at
$0$ and it holds 
$$\mathcal{E}(E'(0)[\beta ]u,v)=\int_{\Omega}\beta uv.$$
\end{lem}

\begin{proof}
Since $\Lambda_{1}\int_{\Omega}u^{2}\le\mathcal{E}(u,u)$, and $\Lambda_{1}>0$,
where $\Lambda_{1}$ is the first eigenvalue of $(-\Delta)^{s}$,
we have $\|E_{\beta}u\|_{L^{2}}\le c\|u\|_{L^{2}}$. Indeed
\[
\Lambda_{1}\int_{\Omega}\left(E_{\beta}u\right)^{2}\le\mathcal{E}(E_{\beta}u,E_{\beta}u)=\int_{\Omega}(\alpha+\beta)uE_{b}u\le c\|u\|_{L^{2}}\|E_{\beta}u\|_{L^{2}}.
\]
We can show now that $\mathcal{E}\left((E_{\beta}-E_{0})u,(E_{\beta}-E_{0})u\right)\rightarrow0$
as $\|\beta\|_{L^{\infty}}\rightarrow0$, proving the continuity of $\beta\rightarrow E_{\beta}$
at $b=0$, in fact $\mathcal{E}\left((E_{\beta}-E_{0})u,w\right)=\int_{\Omega}\beta uw$,
so 
\[
\mathcal{E}\left((E_{\beta}-E_{0})u,(E_{\beta}-E_{0})u\right)=\int_{\Omega}\beta u(E_{\beta}-E_{0})u\le c\|\beta\|_{L^{\infty}}\|u\|_{L^{2}}\mathcal{E}\left((E_{\beta}-E_{0})u,(E_{\beta}-E_{0})u\right)^{\frac{1}{2}}
\]
which proves the claim.

Finally, given $\beta\in C^{0}(\Omega)$ 
and $u\in\mathcal{H}_{0}^{s}$, there exists $L(\beta,u)\in\mathcal{H}_{0}^{s}$
such that
\[
\int_{\Omega}\beta uw=\mathcal{E}\left(L(\beta ,u),w\right).
\]
Thus, for any $w\in\mathcal{H}_{0}^{s}$ it holds
\[
\mathcal{E}\left(\left(E_{\beta}u-E_{0}u-L(\beta,u)\right),w\right)=\int_{\Omega}(\alpha+\beta)uw-\int_{\Omega}\alpha uw-\int_{\Omega}\beta uw\equiv0.
\]
Thus $L(\beta,u)=E'(0)[\beta]u$ and $\mathcal{E}(E'(0)[\beta]u,v)=\int_{\Omega}\beta uv$,
as claimed.
\end{proof}
It remains to us to apply Theorem \ref{thm:astratto} to conclude
the proof of Theorem \ref{main-1}.
\begin{proof}[Proof of Theorem \ref{main-1}.]
 If $\mu^{\alpha}$ is an eigenvalue of multiplicity $\nu>1$ of the operator
$(i^{*}\circ i)_{\alpha}=E_{0}$ and $\varphi_{1}^{\alpha},\dots,\varphi_{\nu}^{\alpha}$
are orthonormal eigenfunctions associated to $\mu^{\alpha}$, the condition
of non splitting is that for any $b$ with $\|\beta\|_{C^{0}}$ small
there exists $\rho=\rho(\beta)\in\mathbb{R}$ such that

\[
\int_{\Omega}\beta\varphi_{i}\varphi_{j}=\rho\delta_{ij},\text{ for all }i,j=1,\dots,\nu.
\]
At this point, the proof can be achieved as the proof of Theorem \ref{main}.
\end{proof}

\section{Appendix}
\begin{proof}[Proof of Lemma \ref{lem:varieta}]
It is known that the Fredholm operator of a given index is open in
$L(X,X)$. So, if $A_{0}\in F_{ij}$, then $A_{0}+H\in F_{ij}$ (if
$H$ is small) if and only if $\dim\left(\ker(A_{0}+H)\right)=\dim\left(\ker(A_{0})\right)$,
that is, if there exists $j$ linearly independent solutions of $(A_{0}+H)x=0$.
By means of the projections $P,Q,\bar{P},\bar{Q}$, this is equivalent
to solve 
\begin{equation}
\left\{ \begin{array}{l}
\bar{P}Hx=0\\
\bar{Q}A_{0}x+\bar{Q}Hx=0
\end{array}\right.;\label{eq:sys}
\end{equation}
Furthermore by Lemma \ref{lem:deco}, we can decompose $H=Y+S+Z+T$
where $Y=\bar{P}HP$, $S=\bar{Q}HP$, $Z=\bar{P}HQ$ and $T=\bar{Q}HQ$.
Set $x=u+v$ where $u\in\ker A_{0}$ and $v\in\mathcal{V}$, we can
recast (\ref{eq:sys}) as
\begin{equation}
\left\{ \begin{array}{l}
Yu+Zv=0\\
\bar{Q}A_{0}v+Su+Tv=0
\end{array}\right..\label{eq:sys-1}
\end{equation}
Now, $\bar{Q}A_{0}:\mathcal{V}\rightarrow\mathrm{Im}A$ is invertible,
and let us call $R$ its inverse. Then the second equation of (\ref{eq:sys-1})
becomes 
\[
v=-RSu-RTv.
\]
If $H$ is sufficiently small, then the operator $w\mapsto-RSu-RTw$
is a contraction from $\mathcal{V}$ to $\mathcal{V}$. Then we can
find $v$ as 
\[
v=-RSu-\sum_{i=0}^{\infty}(-1)^{i}\left(RT\right)^{i}RSu.
\]
Plugging this expression in (\ref{eq:sys-1}) we obtain
\[
\left[Y+Z\left(-RS-\sum_{i=0}^{\infty}(-1)^{i}\left(RT\right)^{i}RS\right)\right]u=0.
\]
Recalling that $u\in\ker A_{0}$, we have that this equation has $j$
linearly independent solutions if and only if 
\[
Y=Z\left(RS+\sum_{i=0}^{\infty}(-1)^{i}\left(RT\right)^{i}RS\right).
\]
Then, when $H$ is small, the set $\left\{ A_{0}+H\in F_{ij}\right\} $
is a graph of an analytic function with domain $\mathcal{V}$, and
the claim follows easily.
\end{proof}


\begin{thebibliography}{1}
\bibitem{BRS}G. M. Bisci, V. D. Radulescu, R. Servadei. \emph{Variational
methods for nonlocal fractional problems.} Vol. 162. Cambridge University
Press, (2016). 


\bibitem{DFW} S. M. Djitte, M. M. Fall, T. Weth,  \emph{
A generalized fractional Pohozaev identity and applications}
 Adv. Calc. Var. https://doi.org/10.1515/acv-2022-0003.  arXiv:2112.10653.



\bibitem{F}
R. Frank,  Eigenvalue bounds for the fractional Laplacian: a review. Recent developments in nonlocal theory, 210--235, De Gruyter, Berlin, 2018.
\bibitem{He}D. Henry, \emph{Perturbation of the boundary in boundary-value
problems of partial differential equations}. London Mathematical Society
Lecture Note Series, 318. Cambridge University Press, Cambridge, 2005.

\bibitem{KWM} T. Kulczycki, M. Kwasnicki, J. Malecki, A. Stos, Spectral properties of the Cauchy process on half-line and interval,
Proc. London Math. Soc. 101 (2) (2010) 589--622

\bibitem{K} M. Kwasnicki, Eigenvalues of the fractional Laplace operator in the interval. J.
Funct. Anal. 262 (2012), no. 5, 2379--2402.

\bibitem{LM}D. Lupo, A.M. Micheletti, \emph{ On the persistence of
the multiplicity of eigenvalues for some variational elliptic operator
depending on the domain}. J. Math. Anal. Appl. \textbf{193} (1995),
no. 3, 990-1002.

\bibitem{M73}A. M. Micheletti, \emph{Perturbazione dello spettro
di un operatore ellittico di tipo variazionale in relazione ad una
variazione di campo} Annali di Matematica Pura ed Applicata (IV) \textbf{(97)
}(1973) 267-282

\bibitem{M}A. M. Micheletti, \emph{Perturbazione dello spettro di
un operatore ellittico di tipo variazionale in relazione ad una variazione
di campo (II)}, Recherche Math. \textbf{25} (1976), 187-200

\bibitem{ROS}X. Ros-Oton, J. Serra, \emph{The Pohozaev identity for the fractional Laplacian}, 
Arch. Ration. Mech. Anal. \textbf{213} (2014), 587-628.

\bibitem{U} K.Uhlenbeck, \emph{
Generic properties of eigenfunctions.}
Amer. J. Math. 98 (1976), no. 4, 1059--1078.
\end{thebibliography}
\end{document}